\documentclass[10pt]{amsart}
\usepackage{enumerate}

\usepackage{color}
\usepackage{palatino}
\usepackage{hyperref}
\usepackage{url}

\usepackage[nice]{nicefrac}

\renewcommand{\baselinestretch}{1.0}
\usepackage{amscd}
\usepackage{amsthm}
\usepackage{amsxtra}
\usepackage{amsfonts,amssymb,amsmath,mathrsfs,graphicx}
\usepackage{float}
\usepackage{epsfig}
\usepackage{graphicx}
\usepackage{verbatim}
\renewcommand{\baselinestretch}{\baselinestretch}
\renewcommand{\baselinestretch}{1.2}
\baselineskip=16pt \hypersetup{
pdftitle={Translating Solitons}, 
pdfauthor={Hojoo Lee}, 
pdfnewwindow=true, 
colorlinks=true, 
linkcolor=black, 
citecolor=blue, 
filecolor=magenta, 
urlcolor=blue 
}
\newtheorem{theorem}{Theorem}
\newtheorem{lemma}[theorem]{Lemma}

\theoremstyle{definition}

\newtheorem{example}[theorem]{Example}
\newtheorem{theorem2}{Theorem}
\theoremstyle{remark}
\newtheorem{remark}[theorem2]{Remark}
\numberwithin{equation}{section}
\begin{document}

\title[The $\mathcal{H}$-flow translating solitons in ${\mathbb{R}}^{3}$ and ${\mathbb{R}}^{4}$]
{The $\mathcal{H}$-flow translating solitons in ${\mathbb{R}}^{3}$
and ${\mathbb{R}}^{4}$}
\author{Hojoo Lee}
\address{Department of Geometry and Topology, University of Granada, Granada, Spain.}
\email{ultrametric@gmail.com}
\thanks{This work was supported by the National Research Foundation of Korea Grant funded
by the Korean Government (Ministry of Education, Science and
Technology) [NRF-2011-357-C00007] and in part by 2010 Korea-France
STAR Program.}
\maketitle
\begin{abstract}
 We solve the prescribed Hoffman-Osserman Gauss map problem for translating soliton  surfaces to the mean
 curvature flow in ${\mathbb{R}}^{4}$. Our solution is inspired by Ilmanen's correspondence between translating soliton
 surfaces and minimal surfaces.
\end{abstract}
\bigskip
 The recent decades admit intensive research devoted to the study of solitons \cite{HS00} to
 the mean curvature flow (${\mathcal{H}}$-flow for short). The simplest example
 is the grim reaper $y = \ln \left( \cos x \right)$
 which moves by downward translation under the ${\mathcal{H}}$-flow.
 As known in \cite{CM11, Il94, Smo01, Wh11}, there exist fascinating geometric dualities between
 the ${\mathcal{H}}$-flow solitons and minimal submanifolds.

  A surface is a \textit{translator} \cite{Wh11} when its mean curvature vector
 field agrees with the normal component of a constant Killing vector field.
 Translators arise as Hamilton's convex eternal solutions and Huisken-Sinestrari's Type II singularities for
 the ${\mathcal{H}}$-flow, and become natural generalization of minimal surfaces.
  The eight equivalent definitions of minimal surfaces illustrated
 in \cite{MP} show the richness of the minimal surfaces theory. However,
 even in ${\mathbb{R}}^{3}$, only few non-minimal translators are known:

Altschuler and Wu \cite{AW94} showed the existence of the convex,
rotationally symmetric, entire graphical translator. Clutterbuck,
Schn\"{u}rer and Schulze \cite{CSS07} constructed the winglike
bigraphical translators, which are analogous to catenoids.
 Halldorsson \cite{Hal11} proved the existence of helicoidal translators. Nguyen \cite{Ng1} used Scherk's minimal towers to desingularize the intersection of a grim reaper cylinder
and a plane, and obtained a complete embedded translator. See also
her generalization \cite{Ng2}.

Our main goal is to adopt the splitting of the generalized Gauss
map of oriented surfaces in ${\mathbb{R}}^{4}$ to construct an
explicit Weierstrass type representation for translators in
${\mathbb{R}}^{4}$. We first introduce the complexification of the
generalized Gauss map. Inside the complex projective space
${\mathbb{C}}{\mathbb{P}}^{3}$, we take the variety
\[
{\mathcal{Q}}_{2}=\{\, [{\zeta}]=[{{\zeta}}_{1}: \cdots:
{{\zeta}}_{4} ] \in {\mathbb{C}}{\mathbb{P}}^{3} :
{{{\zeta}}_{1}}^{2} + \cdots + {{{\zeta}}_{4}}^{2} =0 \,\},
\]
which becomes a model for the Grassmannian manifold
${\mathcal{G}}_{2,2}$ of oriented planes in ${\mathbb{R}}^{4}$.
Reading the biholomorphic map from ${\mathcal{Q}}_{2}$ to
${\mathbb{C}}{\mathbb{P}}^{1} \times {\mathbb{C}}{\mathbb{P}}^{1}$
as a splitting of ${\mathcal{Q}}_{2}$, Hoffman and Osserman
\cite{HO80, HO85} defined the generalized Gauss map of a conformal
immersion ${\mathbf{X}}:\Sigma \rightarrow {\mathbb{R}}^{4}$, $z
\mapsto {\mathbf{X}}(z)$ as follows:
\[
\mathcal{G}(z)=\left[ \; \frac{\partial {\mathbf{X}}}{\partial
{z}} \; \right]= \left[ \; 1+{g}_{1}{g}_{2}, i(1-{g}_{1}{g}_{2}),
{g}_{1}-{g}_{2}, -i\left({g}_{1}+{g}_{2} \right) \; \right] \in
{\mathcal{Q}}_{2} \subset {\mathbb{C}}{\mathbb{P}}^{3}.
\]
We call the induced pair $\left(g_{1}, g_{2}\right)$ the
complexified Gauss map of the immersion ${\mathbf{X}}$.

\begin{lemma}[\textbf{Poincar\'{e}'s Lemma}] \label{PL}
Let $\xi: \Omega \rightarrow \mathbb{C}$ be a function on a simply
connected domain $\Omega \subset \mathbb{C}$. If we have
$\frac{\partial }{\partial \overline{z}} \xi(z) \in \mathbb{R}$
for all $z \in \Omega$, then there exists a function $x: \Omega
\rightarrow \mathbb{R}$ such that $\frac{\partial }{\partial {z}}
x(z)=\xi(z)$.
\end{lemma}

\begin{theorem}[\textbf{Correspondence from null curves in ${\mathbb{C}}^{4}$ to translators in ${\mathbb{R}}^{4}$}] \label{WE}
Let $\left({g}_{1}, g_{2}\right)$ be a pair of nowhere-holomorphic
${\mathcal{C}}^{2}$ functions from a simply connected domain
$\Omega \subset \mathbb{C}$ to the open unit disc $\mathbb{D}:=\{w
\in \mathbb{C} \; \vert \; \vert w \vert < 1 \}$ satisfying the
compatibility condition
\begin{equation} \label{CP}
{\mathcal{F}}:=\frac{{\left({g}_{1}\right)}_{\overline{z}}}{\left(1-g_{1}\overline{g_{2}}\right)\left(1+{\vert
g_{1}\vert}^{2}\right)}
=\frac{{\left({g}_{2}\right)}_{\overline{z}}}{\left(1-\overline{g_{1}}g_{2}\right)\left(1+{\vert
g_{2}\vert}^{2}\right)}, \quad z \in \Omega.
\end{equation}
We assume that one of the following two integrability conditions
holds on $\Omega$:
\begin{equation} \label{Ge1}
0= {\left({g}_{1}\right)}_{z \overline{z}} +
\left(\frac{\overline{g_{2}}}{1-g_{1}\overline{g_{2}}} -
\frac{\overline{g_{1}}}{1+{\vert g_{1}\vert}^{2}}\right)
{\left({g}_{1}\right)}_{z}{\left({g}_{1}\right)}_{\overline{z}}
+\frac{g_{1}+g_{2}}{\left(1-\overline{g_{1}}g_{2}\right)\left(1+{\vert
g_{1}\vert}^{2}\right)} {\vert
{\left({g}_{1}\right)}_{\overline{z}} \vert}^{2},
\end{equation}
\begin{equation} \label{Ge2}
0= {\left({g}_{2}\right)}_{z \overline{z}} +
\left(\frac{\overline{g_{1}}}{1-\overline{g_{1}}g_{2}} -
\frac{\overline{g_{2}}}{1+{\vert g_{2}\vert}^{2}}\right)
{\left({g}_{2}\right)}_{z}{\left({g}_{2}\right)}_{\overline{z}}
+\frac{g_{1}+g_{2}}{\left(1-g_{1}\overline{g_{2}}\right)\left(1+{\vert
g_{2}\vert}^{2}\right)} {\vert
{\left({g}_{2}\right)}_{\overline{z}} \vert}^{2}.
\end{equation}
Then, we obtain the following statements. \\
{\textbf{(a)}} Both (\ref{Ge1}) and (\ref{Ge2}) hold. (Assuming (\ref{CP}), we claim that (\ref{Ge1}) is equivalent to (\ref{Ge2}).)\\
{\textbf{(b)}} The complex curve $\phi:=\left({\phi}_{1},
{\phi}_{2}, {\phi}_{3}, {\phi}_{4} \right) : \Omega \rightarrow
{\mathbb{C}}^{4}$ defined by
\[
\phi=f \left(\, 1 + {g}_{1}{g}_{2}, i(1 - {g}_{1}{g}_{2}), {g}_{1}
- {g}_{2}, -i({g}_{1} + {g}_{2}) \, \right), \quad
f:=-2i{\overline{\mathcal{F}}}
\]
fulfills the three properties on the domain $\Omega$:
\begin{enumerate}
\item[\textbf{(b1)}] \textbf{nullity} \; $\phi \cdot
\phi={{\phi}_{1}}^{2}+{{\phi}_{2}}^{2}+{{\phi}_{3}}^{2}+{{\phi}_{4}}^{2}=0$,
\item[\textbf{(b2)}] \textbf{non-degeneracy} \; ${\vert \phi
\vert}^{2} ={\vert {\phi}_{1} \vert}^{2}+{\vert {\phi}_{2}
\vert}^{2}+{\vert {\phi}_{3} \vert}^{2}+{\vert {\phi}_{4}
\vert}^{2}>0$, \item[\textbf{(b3)}] \textbf{integrability} \;
$\frac{\partial {\phi} }{\partial \overline{z}}= \left(
\frac{\partial {\phi}_{1}}{\partial \overline{z}}, \frac{\partial
{\phi}_{2}}{\partial \overline{z}}, \frac{\partial
{\phi}_{3}}{\partial \overline{z}}, \frac{\partial
{\phi}_{4}}{\partial \overline{z}} \right) \in {\mathbb{R}}^{4}$.
\end{enumerate}
{\textbf{(c)}} Integrating the complex null immersion $\phi:
\Omega \rightarrow {\mathbb{C}}^{4}$ yields a translator $\Sigma$
in ${\mathbb{R}}^{4}$.
\begin{enumerate}
\item[\textbf{(c1)}] There exists a conformal immersion
${\mathbf{X}}=\left( x_{1}, x_{2}, x_{3}, x_{4} \right): \Omega
\rightarrow {\mathbb{R}}^{4}$ satisfying
\[
{\mathbf{X}}_{z} = \phi.
\]
\item[\textbf{(c2)}] The induced metric $ds^{2}$ on the $z$-domain
$\Omega$ by the immersion ${\mathbf{X}}$ reads
\[
ds^{2}= \frac{ 16 \, {\vert {\left({g}_{1}\right)}_{\overline{z}}
\vert}^{2} }{ {\vert 1 - g_{1} \overline{{g}_{2}} \vert}^{2} }
\cdot \frac{ 1+{\vert g_{2}\vert}^{2} }{ 1+{\vert g_{1}\vert}^{2}
} {\vert dz \vert}^{2} = \frac{ 16 \, {\vert
{\left({g}_{2}\right)}_{\overline{z}} \vert}^{2} }{ {\vert 1 -
\overline{{g}_{1}} g_{2} \vert}^{2} } \cdot \frac{ 1+{\vert
g_{1}\vert}^{2} }{ 1+{\vert g_{2}\vert}^{2} } {\vert dz
\vert}^{2}.
\]
\item[\textbf{(c3)}] The pair $\left({g}_{1}, g_{2}\right)$ is the
complexified Gauss map of the surface $\Sigma = {\mathbf{X}}\left(
\Omega \right)$. In other words, the generalized Gauss map of the
conformal immersion ${\mathbf{X}}$ reads
\[
[{\mathbf{X}}_{z} ] = \left[\, 1+{g}_{1}{g}_{2},
i(1-{g}_{1}{g}_{2}), {g}_{1}-{g}_{2},
-i\left({g}_{1}+{g}_{2}\right) \, \right] \in {\mathcal{Q}}_{2}
\subset {\mathbb{C}}{\mathbb{P}}^{3}.
\]
\item[\textbf{(c4)}] The surface $\Sigma$ becomes a translator
with the translating velocity $-\mathbf{e_{4}}=(0,0,0,-1)$.
\end{enumerate}
\end{theorem}

\begin{proof} \textbf{Step A.} For the proof of \textbf{(a)}, we first set up the notations
\[
\begin{cases}
{\mathcal{L}}:= {\left({g}_{1}\right)}_{z \overline{z}} +
\left(\frac{\overline{g_{2}}}{1-g_{1}\overline{g_{2}}} -
\frac{\overline{g_{1}}}{1+{\vert g_{1}\vert}^{2}}\right)
{\left({g}_{1}\right)}_{z}{\left({g}_{1}\right)}_{\overline{z}}
+\frac{g_{1}+g_{2}}{\left(1-\overline{g_{1}}g_{2}\right)\left(1+{\vert g_{1}\vert}^{2}\right)} {\vert {\left({g}_{1}\right)}_{\overline{z}} \vert}^{2}, \\
{\mathcal{R}}:={\left({g}_{2}\right)}_{z \overline{z}} +
\left(\frac{\overline{g_{1}}}{1-\overline{g_{1}}g_{2}} -
\frac{\overline{g_{2}}}{1+{\vert g_{2}\vert}^{2}}\right)
{\left({g}_{2}\right)}_{z}{\left({g}_{2}\right)}_{\overline{z}}
+\frac{g_{1}+g_{2}}{\left(1-g_{1}\overline{g_{2}}\right)\left(1+{\vert
g_{2}\vert}^{2}\right)} {\vert
{\left({g}_{2}\right)}_{\overline{z}} \vert}^{2}.
\end{cases}
\]
We first assume only (\ref{CP}). Taking the conjugation in
(\ref{CP}) yields
\[
\overline{\mathcal{F}}= \frac{ {(\overline{g_{2}})}_{z} } { (1 -
g_{1}\overline{g_{2}}) \left(1+{\vert g_{2}\vert}^{2}\right)} =
\frac{ {(\overline{g_{1}})}_{z} } { (1 - \overline{g_{1}} g_{2})
)\left(1+{\vert g_{1}\vert}^{2}\right) }.
\]
Taking into account this, we deduce
\begin{eqnarray*}
\frac{{\mathcal{F}}_{z}}{{\mathcal{F}}} &=& \frac{{(g_{1})}_{z
\overline{z}}}{{(g_{1})}_{\overline{z}}} +
\left(\frac{\overline{g_{2}}}{1-g_{1}\overline{g_{2}}} -
\frac{\overline{g_{1}}}{1+{\vert g_{1}\vert}^{2}}\right)
{\left({g}_{1}\right)}_{z}
+ \left( \frac{ 1+{\vert g_{2}\vert}^{2} }{1- \overline{g_{1}}g_{2}} - 1 \right) \cdot \frac{g_{1}}{ 1+{\vert g_{1}\vert}^{2} } {(\overline{g_{1}})}_{z} \\
&=& \frac{{\mathcal{L}}}{{(g_{1})}_{\overline{z}}} - \overline{F}
\left[ g_{1} \left( 1-{\vert g_{2}\vert}^{2} \right) + g_{2}
\left( 1-{\vert g_{1}\vert}^{2} \right) \right]
\end{eqnarray*}
and
\begin{eqnarray*}
\frac{{\mathcal{F}}_{z}}{{\mathcal{F}}} &=& \frac{{(g_{2})}_{z
\overline{z}}}{{(g_{2})}_{\overline{z}}} +
\left(\frac{\overline{g_{1}}}{1-g_{2}\overline{g_{2}}} -
\frac{\overline{g_{2}}}{1+{\vert g_{2}\vert}^{2}}\right)
{\left({g}_{2}\right)}_{z}
+ \left( \frac{ 1+{\vert g_{1}\vert}^{2} }{1- g_{1} \overline{g_{2}}} - 1 \right) \cdot \frac{g_{2}}{ 1+{\vert g_{2}\vert}^{2} } {(\overline{g_{2}})}_{z} \\
&=& \frac{{\mathcal{R}}}{{(g_{2})}_{\overline{z}}} - \overline{F}
\left[ g_{1} \left( 1-{\vert g_{2}\vert}^{2} \right) + g_{2}
\left( 1-{\vert g_{1}\vert}^{2} \right) \right].
\end{eqnarray*}
These two equalities thus show the equality
\[
\frac{{\mathcal{L}}}{{(g_{1})}_{\overline{z}}} =
\frac{{\mathcal{R}}}{{(g_{2})}_{\overline{z}}},
\]
which means the desired implications: (\ref{Ge1})
$\Longleftrightarrow \mathcal{L}=0
\Longleftrightarrow \mathcal{R}=0 \Longleftrightarrow $ (\ref{Ge2}). \\
\textbf{Step B.} We deduce several equalities which will be used
in the proof of \textbf{(b)} and \textbf{(c)}. According to
\textbf{(a)}, from now on, we assume that both (\ref{Ge1}) and
(\ref{Ge2}) hold. Since both $\mathcal{L}$ and $\mathcal{R}$
vanish, the previous equalities imply
\[
{\mathcal{F}}_{z} = - {\vert \mathcal{F} \vert}^{2} \left[ g_{1}
\left( 1-{\vert g_{2}\vert}^{2} \right) + g_{2} \left( 1-{\vert
g_{1}\vert}^{2} \right) \right].
\]
Conjugating this and using the definition
$f=-2i{\overline{\mathcal{F}}}$, we arrive at the equality
\begin{equation} \label{Tq1}
{f}_{\overline{z}} = \frac{i}{2} {\vert f \vert}^{2} \left[
\overline{g_{1}} \left( 1-{\vert g_{2}\vert}^{2} \right) +
\overline{g_{2}} \left( 1-{\vert g_{1}\vert}^{2} \right) \right].
\end{equation}
The compatibility condition (\ref{CP}) can be written in terms of
$f=-2i{\overline{\mathcal{F}}}$:
\begin{equation} \label{CPf}
\overline{f}
=\frac{2i{\left({g}_{1}\right)}_{\overline{z}}}{\left(1-g_{1}\overline{g_{2}}\right)\left(1+{\vert
g_{1}\vert}^{2}\right)}
=\frac{2i{\left({g}_{2}\right)}_{\overline{z}}}{\left(1-\overline{g_{1}}g_{2}\right)\left(1+{\vert
g_{2}\vert}^{2}\right)}.
\end{equation}
It immediately follows from (\ref{Tq1}) and (\ref{CPf}) that
\begin{equation} \label{Tq2a}
{\left(fg_{1} \right)}_{\overline{z}} = {f}_{\overline{z}} g_{1} +
{\left(g_{1} \right)}_{\overline{z}} f = - \frac{i}{2} {\vert f
\vert}^{2} \left( 1 - 2 g_{1} \overline{g_{2}} + {\vert
g_{1}\vert}^{2} {\vert g_{2}\vert}^{2} \right)
\end{equation}
and
\begin{equation} \label{Tq2b}
{\left(fg_{2} \right)}_{\overline{z}} ={f}_{\overline{z}} g_{2} +
{\left(g_{2} \right)}_{\overline{z}} f = - \frac{i}{2} {\vert f
\vert}^{2} \left( 1 - 2 \overline{g_{1}} g_{2} + {\vert
g_{1}\vert}^{2} {\vert g_{2}\vert}^{2} \right).
\end{equation}
Another computation taking into account (\ref{CPf}) and
(\ref{Tq2a}) shows
\begin{equation} \label{Tq3}
{\left(fg_{1}g_{2} \right)}_{\overline{z}} = {\left(fg_{1}
\right)}_{\overline{z}} g_{2} + {\left(g_{2}
\right)}_{\overline{z}} fg_{1} = - \frac{i}{2} {\vert f \vert}^{2}
\left[ {g_{1}} \left( 1-{\vert g_{2}\vert}^{2} \right) + {g_{2}}
\left( 1-{\vert g_{1}\vert}^{2} \right) \right].
\end{equation}
\textbf{Step C.} Our aim here is to establish the claims in
\textbf{(b)} on the complex curve
\[
\phi=\left({\phi}_{1}, {\phi}_{2}, {\phi}_{3}, {\phi}_{4}
\right)=f \left(\, 1 + {g}_{1}{g}_{2}, i(1 - {g}_{1}{g}_{2}),
{g}_{1} - {g}_{2}, -i({g}_{1} + {g}_{2}) \, \right).
\]
First, the equality in \textbf{(b1)} is obvious. Next, by the
assumptions on $g_{1}$ and $g_{2}$, we see that
$f=-2i{\overline{\mathcal{F}}}$ never vanish. Then, the assertion
\textbf{(b2)} follows from the equality
\begin{equation} \label{met1}
{\vert {\phi}_{1} \vert}^{2}+{\vert {\phi}_{2} \vert}^{2}+{\vert
{\phi}_{3} \vert}^{2}+{\vert {\phi}_{4} \vert}^{2} = 2 {\vert f
\vert}^{2} \left( 1 + {\vert g_{1}\vert}^{2} \right) \left( 1 +
{\vert g_{2}\vert}^{2} \right).
\end{equation}
We employ the equalities in \textbf{Step B} to show the assertion
\textbf{(b3)}. Joining the equalities in (\ref{Tq1}),
(\ref{Tq2a}), (\ref{Tq2b}), and (\ref{Tq3}) and the definition of
$\phi$, we reach
\begin{equation} \label{tqs}
\begin{cases}
{\left({\phi}_{1}\right)}_{\overline{z}} = \;\;\, {\vert f
\vert}^{2} \left[ \left( 1 - {\vert g_{2} \vert}^{2} \right)
\mathrm{Im }\,g_{1}
+ \left( 1 - {\vert g_{1} \vert}^{2} \right) \mathrm{Im }\,g_{2} \right], \\
{\left({\phi}_{2}\right)}_{\overline{z}} = -{\vert f \vert}^{2}
\left[ \left( 1 - {\vert g_{2} \vert}^{2} \right) \mathrm{Re
}\,g_{1}
+ \left( 1 - {\vert g_{1} \vert}^{2} \right) \mathrm{Re }\,g_{2} \right], \\
{\left({\phi}_{3}\right)}_{\overline{z}}
= \, 2 {\vert f \vert}^{2} \; \mathrm{Im }\, \left( \overline{g_{1}} g_{2} \right), \\
{\left({\phi}_{4}\right)}_{\overline{z}} = -{\vert f \vert}^{2}
\left[ 1 - 2 \mathrm{Re }\, \left( \overline{g_{1}} g_{2} \right)
+ {\vert g_{1} \vert}^{2} {\vert g_{2} \vert}^{2} \right].
\end{cases}
\end{equation}
These four equalities guarantee the integrability condition
$\left( \frac{\partial {\phi}_{1}}{\partial \overline{z}}, \frac{\partial {\phi}_{2}}{\partial \overline{z}}, \frac{\partial {\phi}_{3}}{\partial \overline{z}}, \frac{\partial {\phi}_{4}}{\partial \overline{z}} \right) \in {\mathbb{R}}^{4}$.\\
\textbf{Step D.} We prove the claims \textbf{(c1)}, \textbf{(c2)},
and \textbf{(c3)}. Thanks to \textbf{(b3)}, we can integrate the
curve $\phi$. Since $\Omega$ is simply connected, applying Lemma
\ref{PL} to the complex curve $\phi$, we see the existence of the
function $ {\mathbf{X}}=\left( x_{1}, x_{2}, x_{3}, x_{4} \right):
\Omega \rightarrow {\mathbb{R}}^{4}$ satisfying
\[
{\mathbf{X}}_{z} = \phi=f \left(\, 1 + {g}_{1}{g}_{2}, i(1 -
{g}_{1}{g}_{2}), {g}_{1} - {g}_{2}, -i({g}_{1} + {g}_{2}) \,
\right).
\]
This and the nullity of $\phi$ guarantee that the mapping $
{\mathbf{X}}$ is conformal. Using (\ref{met1}), one then find that
the induced metric $ds^{2}= {\Lambda}^{2} {\vert dz\vert}^{2}$ by
the immersion ${\mathbf{X}}$ reads
\begin{equation} \label{met2}
ds^{2}= {\Lambda}^{2} {\vert dz\vert}^{2}= 4 {\vert f \vert}^{2}
\left( 1 + {\vert g_{1}\vert}^{2} \right) \left( 1 + {\vert
g_{2}\vert}^{2} \right) {\vert dz\vert}^{2}.
\end{equation}
Since $f$ never vanish, this completes the proof of \textbf{(c1)}.
Also, joining (\ref{CPf}) and (\ref{met2}) imply the equality in
\textbf{(c2)}. The integrability ${\mathbf{X}}_{z} = \phi$ and the
definition of $\phi$ gives
\[
[ {\mathbf{X}}_{z} ] = \left[ 1+{g}_{1}{g}_{2},
i(1-{g}_{1}{g}_{2}), {g}_{1}-{g}_{2},
-i\left({g}_{1}+{g}_{2}\right)\right],
\]
which completes the proof of \textbf{(c3)}. \\
\textbf{Step E.} Finally, we prove the claim \textbf{(c4)}. First,
we find the normal component of the vector field
$-\mathbf{e_{4}}=(0,0,0,-1)$ in terms of $g_{1}$ and $g_{2}$. We
compute
\begin{eqnarray*}
{\left(-\mathbf{e_{4}}\right)}^{\perp} &=& -\mathbf{e_{4}} -
\left[ \left( \frac{ {\mathbf{X}}_{u}}{\Lambda} \cdot
\left(-\mathbf{e_{4}} \right) \right) \frac{
{\mathbf{X}}_{u}}{\Lambda}
+ \left( \frac{ {\mathbf{X}}_{v}}{\Lambda} \cdot \left(-\mathbf{e_{4}} \right) \right) \frac{ {\mathbf{X}}_{v}}{\Lambda} \right] \\
&=& -\mathbf{e_{4}} + \frac{2}{{\Lambda}^{2}} \left[
\left( {\mathbf{X}}_{\overline{z}} \cdot \mathbf{e_{4}} \right) {\mathbf{X}}_{z} + \left( {\mathbf{X}}_{z} \cdot \mathbf{e_{4}} \right) {\mathbf{X}}_{\overline{z}} \right] \\
&=& \begin{bmatrix}
0 \\
0\\
0 \\
-1
\end{bmatrix}
+ \frac{4}{{\Lambda}^{2}}
\begin{bmatrix}
\mathrm{Re }\, \left( {\phi}_{1} \overline{{\phi}_{4}} \right) \\
\mathrm{Re }\, \left( {\phi}_{2} \overline{{\phi}_{4}} \right) \\
\mathrm{Re }\, \left( {\phi}_{3} \overline{{\phi}_{4}} \right) \\
{\vert {\phi}_{4} \vert}^{2}.
\end{bmatrix}.
\end{eqnarray*}
Combining this,  (\ref{tqs}), and (\ref{met2}) yields
\[
{\left(-\mathbf{e_{4}}\right)}^{\perp}= \frac{1}{\left( 1 + {\vert
g_{1}\vert}^{2} \right) \left( 1 + {\vert g_{2}\vert}^{2} \right)}
\begin{bmatrix}
\;\, \left[ \left( 1 - {\vert g_{2} \vert}^{2} \right) \mathrm{Im
}\,g_{1}
+ \left( 1 - {\vert g_{1} \vert}^{2} \right) \mathrm{Im }\,g_{2} \right] \\
-  \left[ \left( 1 - {\vert g_{2} \vert}^{2} \right) \mathrm{Re
}\,g_{1}
+ \left( 1 - {\vert g_{1} \vert}^{2} \right) \mathrm{Re }\,g_{2} \right] \\
2   \; \mathrm{Im }\, \left( \overline{g_{1}} g_{2} \right) \\
-  \left[ 1 - 2 \mathrm{Re }\, \left( \overline{g_{1}} g_{2}
\right) + {\vert g_{1} \vert}^{2} {\vert g_{2} \vert}^{2} \right]
\end{bmatrix}.
\]
Second, we find the mean curvature vector
$\mathcal{H}={\triangle}_{ds^{2}} {\mathbf{X}}= \frac{4}{
{\Lambda}^{2} } \frac{\partial }{\partial \overline{z}} \left(
\frac{\partial }{\partial z} {\mathbf{X}} \right) = \frac{4}{
{\Lambda}^{2} } {\phi}_{\overline{z}}$ on our surface
$\Sigma=\mathbf{X}\left( \Omega \right)$. Now, joining this,
(\ref{tqs}), and (\ref{met2}), we can write the mean curvature
vector $\mathcal{H}$ in terms of $g_{1}$ and $g_{2}$:
\[
\mathcal{H}= \frac{1}{\left( 1 + {\vert g_{1}\vert}^{2} \right)
\left( 1 + {\vert g_{2}\vert}^{2} \right)}
\begin{bmatrix}
\;\,  \left[ \left( 1 - {\vert g_{2} \vert}^{2} \right) \mathrm{Im
}\,g_{1}
+ \left( 1 - {\vert g_{1} \vert}^{2} \right) \mathrm{Im }\,g_{2} \right] \\
-  \left[ \left( 1 - {\vert g_{2} \vert}^{2} \right) \mathrm{Re
}\,g_{1}
+ \left( 1 - {\vert g_{1} \vert}^{2} \right) \mathrm{Re }\,g_{2} \right] \\
2   \; \mathrm{Im } \, \left( \overline{g_{1}} g_{2} \right) \\
-  \left[ 1 - 2 \mathrm{Re }\, \left( \overline{g_{1}} g_{2}
\right) + {\vert g_{1} \vert}^{2} {\vert g_{2} \vert}^{2} \right]
\end{bmatrix}.
\]
We therefore conclude that
$\mathcal{H}={\left(-\mathbf{e_{4}}\right)}^{\perp}$.
\end{proof}

\begin{remark}[\textbf{Ilmanen's correspondence}] \label{Ilc} Theorem \ref{WE} generalizes the classical Weierstrass construction from holomorphic null immersions in ${\mathbb{C}}^{3}$ to conformal minimal immersions in ${\mathbb{R}}^{3}$. The key ingredient behind Theorem \ref{WE} is the Ilmanen correspondence between translators and minimal surfaces. (See
\cite{Il94} and \cite{Wh11}.) We deform the flat metric of
${\mathbb{R}}^{4}$ conformally to introduce the four dimensional
Riemannian manifold
\[
{\mathcal{I}}^{4}=\left(\, {\mathbb{R}}^{4}, \,
e^{-x_{4}}\left({dx_{1}}^{2}+
{dx_{2}}^{2}+{dx_{3}}^{2}+{dx_{4}}^{2}\right) \, \right).
\]
Any conformal immersion ${\mathbf{X}}: \Omega \rightarrow
{\mathbb{R}}^{4}$ of a downward translator with the translating
velocity $-\mathbf{e_{4}}=(0,0,0,-1)$ in Euclidean space
${\mathbb{R}}^{4}$ then can be identified as a conformal minimal
immersion ${\mathbf{X}}: \Omega \rightarrow {\mathcal{I}}^{4}$.
However, it is not easy to find Riemannian manifolds which admit
explicit representations for their minimal surfaces.
\end{remark}

\begin{example}[\textbf{The Hamiltonian stationary Lagrangian translator in ${\mathbb{C}}^{2}$}]
 Recently, interesting Lagrangian translators in the complex plane ${\mathbb{C}}^{2}$ are discovered in \cite{CL10, JLT10, Na11}.
 In 2010, Castro and Lerma \cite[Corollary 2]{CL10} classified all Hamiltonian stationary Lagrangian
  translators in ${\mathbb{C}}^{2}$. Locally, they are unique up to dilations (except
 for the totally geodesic ones) \cite[Corollary 3]{CL10}. The point of this example is to explicitly recover
 the Hoffman-Osserman Gauss map of the Castro-Lerma translator in $\mathbb{R}^{4}={\mathbb{C}}^{2}$.

 We first notice that Theorem \ref{WE} still holds when we regard the prescribed Gauss map $\left(g_{1}, g_{2}\right)$ as
 a pair of functions from a simply connected domain $\Omega$ to the complex plane (not just the unit disc). However, in this
 case, the induced mapping ${\mathbf{X}}: \Omega \rightarrow {\mathbb{R}}^{4}$ of the translator may admit the branch points
 where $\overline{g_{1}}g_{2}=1$ (or equivalently, $g_{1}\overline{g_{2}}=1$).

 Imposing the additional condition $\vert g_{1} \vert=1$ produces Lagrangian translators with the
  velocity $-\mathbf{e_{4}}=\left( 0,0,0,-1 \right)$. Then, our integrability condition
  in \textbf{(c1)}  for downward translators can be re-written as
\[
 {\mathbf{X}}_{z} = \left( \, {(x_{1})}_{z}, {(x_{2})}_{z}, {(x_{3})}_{z}, {(x_{4})}_{z} \, \right)
 = - {\theta}_{z} \left(\,\frac{1+g_{1}g_{2}}{g_{1}-g_{2}}, \,i \frac{1-g_{1}g_{2}}{g_{1}-g_{2}},
  \,1, \,- i \frac{g_{1}+ g_{2}}{g_{1}-g_{2}}\, \right),
\]
 where $\theta$ denotes the Lagrangian angle with $ig_{1}=e^{i\theta}$. The third terms ${(x_{3})}_{z}=-{\theta}_{z}$ can be
 viewed as \cite[Proposition 1]{CL10}, \cite[Proposition 2.5]{JLT10} and \cite[Proposition 2.1]{NT07}.

 We consider a complexified Gauss map of the form, for some $\mathbb{R}$-valued function $\mathcal{G}$,
\[
 \left(\, g_{1}(z), \, g_{2}(z) \, \right) = \left(\, e^{iv}, \, \mathcal{G}(u) e^{iv}
 \, \right),\quad z=u+iv \in \mathbb{R}+i\mathbb{R}
\]
 and want to solve the system (\ref{CP}) and (\ref{Ge1}). First, the compatibility condition (\ref{CP}) induces the
 ordinary differential equation
 \[
 \frac{1}{2} = \frac{1}{1+{\mathcal{G}}^{2}} \left( \mathcal{G} - \frac{d\mathcal{G}}{du} \right),
 \]
  and a canonical solution is given by $\mathcal{G}(u)=\frac{u+1}{u-1}$. One can easily check that
 \[
  \left(\, g_{1}(z), \, g_{2}(z) \, \right) = \left(\, e^{iv}, \, \mathcal{G}(u) e^{iv} \, \right)= \left( e^{iv}, \frac{u+1}{u-1} e^{iv} \right)
 \]
  satisfies the integrability condition (\ref{Ge1}). Then, the induced Lagrangian translator $\Sigma$ with
  the velocity $-\mathbf{e_{4}}$ admits the conformal parametrization
\[
 {\mathbf{X}}(u, v) = \left(    u \sin v, - u \cos v, -v, - \frac{1}{2} u^2 \right).
\]
 Since the induced metric  on $\Sigma$ reads $ds^2 = \left( 1+ u^2 \right) \left( du^2 + dv^2 \right)$,
 the  Lagrangian angle function  $\theta(u,v) = \frac{\pi}{2} + v$ with $ig_{1}=e^{i\theta}$ is harmonic on $\Sigma$.
 We find that this Hamiltonian stationary Lagrangian translator $\Sigma$ with the velocity $\left( 0,0,0,-1 \right)$ coincides with the
 Castro-Lerma translator \cite[Corollary 2]{CL10} with the velocity $\left(1,0,0,0 \right)$ by a suitable change of coordinates.
\end{example}

\begin{theorem}[\textbf{Correspondence from null curves in ${\mathbb{C}}^{3}$ to translators in ${\mathbb{R}}^{3}$}] \label{WEIER}
When a nowhere-holomorphic ${\mathcal{C}}^{2}$ function $G:\Omega
\rightarrow \mathbb{D}$ from a simply connected domain $\Omega
\subset \mathbb{C}$ to the open unit disc $\mathbb{D}:=\{w \in
\mathbb{C} \; \vert \; \vert w \vert < 1 \}$ satisfying the
translator equation
\begin{equation} \label{Gauss}
{G}_{z \overline{z}} + 2 \frac{ \overline{G} \, {\vert G
\vert}^{2}}{ 1 -{\vert G \vert}^{4} } G_{z} G_{\overline{z}} + 2
\frac{ G }{ 1 -{\vert G \vert}^{4}} {\vert G_{\overline{z}}
\vert}^{2}=0, \quad z \in \Omega,
\end{equation}
we associate a complex curve
$\phi={\phi}_{{}_{G}}=\left({\phi}_{1}, {\phi}_{2},
{\phi}_{3}\right) : \Omega \rightarrow {\mathbb{C}}^{3}$ as
follows:
\[
\phi = \frac{2 {\overline{G}}_{z} }{ {\vert G \vert}^{4} -1 }
\left( 1-{G}^{2}, i \left( 1 + G^{2} \right), 2 G \right).
\]
{\textbf{(a)}} Then, the complex curve $\phi$ fulfills the three
properties on the domain $\Omega$:
\begin{enumerate}
\item[\textbf{(a1)}] \textbf{nullity} \; $\phi \cdot
\phi={{\phi}_{1}}^{2} + {{\phi}_{2}}^{2} +{{\phi}_{3}}^{2}=0$,
\item[\textbf{(a2)}] \textbf{non-degeneracy} \; ${\vert \phi
\vert}^{2} ={\vert {\phi}_{1} \vert}^{2}+{\vert {\phi}_{2}
\vert}^{2}+{\vert {\phi}_{3} \vert}^{2}>0$, \item[\textbf{(a3)}]
\textbf{integrability} \; $\frac{\partial {\phi} }{\partial
\overline{z}}= \left( \frac{\partial {\phi}_{1}}{\partial
\overline{z}}, \frac{\partial {\phi}_{2}}{\partial \overline{z}},
\frac{\partial {\phi}_{3}}{\partial \overline{z}} \right) \in
{\mathbb{R}}^{3}$.
\end{enumerate}
{\textbf{(b)}} Also, integrating ${\mathbf{X}}_{z} = \phi$ on
$\Omega$ yields a downward translator $\Sigma =\mathbf{X}\left(
\Omega \right)$  with the velocity $-\mathbf{e_{3}}=(0,0,-1)$ in
${\mathbb{R}}^{3}$.  The prescribed map $G$ becomes the
complexified Gauss map of the induced surface $\Sigma
=\mathbf{X}\left( \Omega \right)$ via the stereographic projection
from the north pole.  The induced metric $ds^{2}$ by the immersion
${\mathbf{X}}$ reads $ds^{2}= \frac{ 16 \, {\vert
{G}_{\overline{z}} \vert}^{2} }{ {\left( {\vert G \vert}^{2} - 1
\right)}^{2} }{\vert dz \vert}^{2}$.
\end{theorem}

\begin{proof}
We take $(g_{1}, g_{2})=(iG, iG)$ in Theorem \ref{WE}.
\end{proof}

\begin{example}[\textbf{Downward grim reaper cylinder as an analogue of Scherk's surface}] \label{grc}
${}_{}$ \\
\textbf{(a)} An application of our representation formula in
Theorem \ref{WEIER} to the solution
\[
G(z)=G(u+iv)=\tanh u \in (-1,1), \quad u+iv \in {\mathbb{C}}
\]
 of the translator equation (\ref{Gauss})
yields the conformal immersion $\mathbf{X}: {\mathbb{R}}^{2}
\rightarrow {\mathbb{R}}^{3}$
\[
\mathbf{X}(u,v)=\left(x_{1}, x_{2}, x_{3} \right) = \left( -2 \,
{\tan}^{-1} \left( \tanh u \right), \, 2v, \, - \ln \left( \cosh
(2u) \right) \, \right),
\]
\textbf{(b)} It represents the graphical translator with the
translating velocity $-\mathbf{e_{3}}$:
\[
x_{3} = \mathcal{F} \left( x_{1}, x_{2} \right)= \ln \left( \cos
x_{1} \right), \quad \left( x_{1}, x_{2} \right) \in \left( -
\frac{\pi}{2} , \frac{\pi}{2} \right) \times {\mathbb{R}}.
\]
Its height function $\mathcal{F}: \left( -\frac{\pi}{2},
\frac{\pi}{2} \right) \times \mathbb{R} \rightarrow \mathbb{R}$ is
a Jenkins-Serrin type solution of
\begin{equation} \label{gt}
\nabla \cdot \left( \frac{1}{\sqrt{ 1 + {\vert \nabla \mathcal{F}
\vert}^{2} }} \nabla \mathcal{F} \right) + \frac{ 1}{\sqrt{ 1 +
{\vert \nabla \mathcal{F} \vert}^{2} }} = 0
\end{equation}
and has $-\infty$ boundary values. Our graph $x_{3} = \mathcal{F}
\left( x_{1}, x_{2} \right)$ becomes a cylinder over the downward
grim reaper on the $x_{1}x_{3}$-plane. It can be viewed as an
analogue of the classical Jenkins-Serrin type minimal graph,
discovered by Scherk in 1834,
\[
x_{3} = \ln \left( \cos x_{1} \right)-\ln \left( \cos x_{2}
\right), \quad \left( x_{1}, x_{2} \right) \in \left( -
\frac{\pi}{2} , \frac{\pi}{2} \right) \times \left( -
\frac{\pi}{2} , \frac{\pi}{2} \right).
\]
Its height function takes the values $\pm \infty$ on alternate sides of the square domain. \\
\textbf{(c)} More generally, Scherk discovered the doubly periodic
minimal graph ${\Sigma}_{\rho}^{2\alpha}$ \cite{Nit88}:
\[
x_{3} = \frac{1}{\rho} \left[\, \ln \, \left( \cos \left(
\frac{\rho}{2} \left[ \frac{x_{1}}{\cos \alpha} -
\frac{x_{2}}{\sin \alpha} \right] \, \right) \right) - \ln \left(
\, \cos \left( \frac{\rho}{2} \left[ \frac{x_{1}}{\cos \alpha} +
\frac{x_{2}}{\sin \alpha} \right] \, \right) \right) \, \right]
\]
for some constants $\alpha \in \left(0, \frac{\pi}{2} \right)$ and
$\rho > 0$. It is defined on an infinite chess board-like net of
\textit{rhomboids}. Its picture is available at \cite{Web}.
However, unlike the deformations of Scherk's minimal surfaces by
\textit{shearing}, it is not possible to shear the grim reaper
cylinder to obtain non-trivial deformations of unit-speed
translators.
\end{example}

\begin{remark}[\textbf{Jenkins-Serrin type problem for graphical translators}] \label{IB}
A beautiful theory for infinite boundary value problems of minimal
graphs is developed by Jenkins and Serrin \cite{JS66}. Moreover,
Spruck \cite{Spr72} obtained a Jenkins-Serrin type theory for
constant mean curvature graphs. It would be very interesting to
investigate a similar Dirichlet problem for graphical translators.
In the following Example \ref{grctheta}, we prove that, for any $l
\geq \pi$, there exists a downward \textit{unit-speed}  graphical
translator that its height function is defined over an infinite
strip of width $l$ and takes the values $-\infty$ on its boundary.
We propose a conjecture that the lower bound $\pi$ is a critical
constant in the sense that, for any $l \in (0, \pi)$, there exists
no downward \textit{unit-speed}  graphical translator defined over
an infinite strip of width $l$ approaching $-\infty$ on its
boundary.
\end{remark}

\begin{example}[\textbf{Deformations of grim reaper cylinder}] \label{grctheta}
Let $\theta \in \mathbb{R}$ be a constant. \\
\textbf{(a)} We begin with the following solution
$G=G^{\theta}(z)$ of the translator equation (\ref{Gauss}):
\[
G(z)=G(u+iv)=\frac{\cosh \theta \sinh (2u) + i \sinh \theta}{1+
\cosh \theta \cosh (2u)} , \quad u+iv \in {\mathbb{C}}.
\]
Theorem \ref{WEIER} then induces the conformal immersion
${\mathbf{X}}^{\theta}=\left(x_{1}, x_{2}, x_{3} \right):
{\mathbb{R}}^{2} \rightarrow {\mathbb{R}}^{3}$
\[
\begin{cases}
x_{1} (u,v) = -2 \cosh \theta \, {\tan}^{-1} \left( \tanh u \right), \\
x_{2} (u,v) = \sinh \theta \ln \left( \cosh (2u) \right) + 2v, \\
x_{3} (u,v) = - \ln \left( \cosh (2u) \right) + 2v \sinh \theta.
\end{cases}
\]
The downward translator ${\mathbb{G}}^{\theta}={\mathbf{X}}^{\theta}\left({\mathbb{R}}^{2}\right)$ has the translating velocity $(0,0,-1)$. \\
\textbf{(b)} Using the patch ${\mathbf{X}}^{\theta}$, one can
easily check that Gauss map of the translator
${\mathbb{G}}^{\theta}$ lies on a half circle. Let us introduce a
new linear coordinate
\[
x_{0}=\frac{1}{\cosh \theta} x_{2}+\frac{\sinh \theta}{\cosh
\theta} x_{3}
\]
and then prepare an orthonormal basis
\[
{\mathcal{U}}_{1}=\left(1,0,0\right), \;
{{\mathcal{U}}_{2}}^{\theta} =\left(0,- \frac{\sinh \theta}{\cosh
\theta},\frac{1}{\cosh \theta}\right), \;
{{\mathcal{U}}_{3}}^{\theta}=\left(0,\frac{1}{\cosh
\theta},\frac{\sinh \theta}{\cosh \theta}\right).
\]
It is easily shown that our surface ${\mathbb{G}}^{\theta}$ admits
a new geometric patch
\[
\left(x_{1}, x_{2}, x_{3} \right) =
{\widehat{\mathbf{X}}}^{\theta}\left(x_{1},x_{0}\right) = x_{1} \,
{\mathcal{U}}_{1} + {\mathbf{T}}^{\theta}(x_{1}) \,
{{\mathcal{U}}_{2}}^{\theta} + x_{0} \,
{{\mathcal{U}}_{3}}^{\theta}.
\]
Here, ${\mathbf{T}}^{\theta}(\cdot)=\cosh \theta \ln \, \left(
\cos \left( \frac{\cdot}{\cosh \theta} \right) \, \right)$ is a
parabolic rescaling of the downward unit-speed grim reaper
function. The patch ${\widehat{\mathbf{X}}}^{\theta}$ says that
the surface ${\mathbb{G}}^{\theta}$ becomes a cylinder over a
parabolically rescaled grim reaper curve in the plane
spanned by ${\mathcal{U}}_{1}$ and ${{\mathcal{U}}_{2}}^{\theta}$. \\
\textbf{(c)} Our one-parameter family
$\{{\mathbb{G}}^{\theta}\}_{\theta \in {\mathbb{R}}}$ of cylinders
with the same translating velocity admits a simple geometric
description. Applying a suitable rotation in the ambient space
${\mathbb{R}}^{3}$ to the grim reaper cylinder ${\mathbb{G}}^{0}$
with velocity $-{{\mathcal{U}}_{2}}^{0}=(0,0,-1)$:
\[
  \left(x_{1}, x_{0}\right) \in \left( - \frac{\pi}{2}, \frac{\pi}{2} \right) \times \mathbb{R}
  \; \mapsto \; {\widehat{\mathbf{X}}}^{0}\left(x_{1},x_{0}\right) = x_{1} \, {\mathcal{U}}_{1} + {\mathbf{T}}^{0}(x_{1}) \, {{\mathcal{U}}_{2}}^{0} + x_{0} \, {{\mathcal{U}}_{3}}^{0},
\]
we obtain the congruent cylinder parametrized by
\[
  \left(x_{1}, x_{0}\right) \in \left( - \frac{\pi}{2}, \frac{\pi}{2} \right) \times \mathbb{R}
  \; \mapsto \;  x_{1} \, {\mathcal{U}}_{1} + {\mathbf{T}}^{0}(x_{1}) \, {{\mathcal{U}}_{2}}^{\theta} + x_{0} \, {{\mathcal{U}}_{3}}^{\theta},
\]
which translates with the rotated velocity
$-{{\mathcal{U}}_{2}}^{\theta}$ under the ${\mathcal{H}}$-flow.
However, we observe that this rotated cylinder also can be viewed
as a translator with new velocity  $- \cosh \theta \,
{{\mathcal{U}}_{2}}^{0}=(0,0,-\cosh \theta)$. Employing the
appropriate parabolic rescaling as a speed-down action, we meet
our cylinder ${\mathbb{G}}^{\theta}$ parametrized by
\[
  \left(x_{1}, x_{0}\right) \in \left( -\frac{\pi}{2}\cosh \theta, \frac{\pi}{2}\cosh \theta \right)  \times \mathbb{R}
  \; \mapsto \; {\widehat{\mathbf{X}}}^{\theta}\left(x_{1},x_{0}\right) = x_{1} \, {\mathcal{U}}_{1} + {\mathbf{T}}^{\theta}(x_{1}) \, {{\mathcal{U}}_{2}}^{\theta} + x_{0} \, {{\mathcal{U}}_{3}}^{\theta},
\]
which translates with velocity $-{{\mathcal{U}}_{2}}^{0}=(0,0,-1)$ under the ${\mathcal{H}}$-flow. \\
\textbf{(d)} We prove the claim in Remark \ref{IB}. We are able to
view the downward unit-speed translator ${\mathbb{G}}^{\theta}$ as
the graph of the function ${\mathcal{F}}^{\theta}: \left(
-\frac{\pi}{2}\cosh \theta, \frac{\pi}{2}\cosh \theta \right)
\times \mathbb{R} \rightarrow \mathbb{R}$:
\[
x_{3}={\mathcal{F}}^{\theta}\left(x_{1},x_{2}\right) = \cosh
\theta \, {\mathbf{T}}^{\theta}( x_{1} ) + \sinh \theta \; x_{2}.
\]
Its height function ${\mathcal{F}}^{\theta}$ solves the PDE
(\ref{gt}) over the strip of width $l=\pi \cosh \theta \geq \pi$.
\end{example}

 \textbf{Acknowledgement.} I would like to thank Brian White for letting me know about Tom Ilmanen's paper \cite{Il94}, where
 he presented the correspondence in Remark \ref{Ilc}. I also would like to thank Ildefonso Castro for his interests in this
 work and helpful comments on the Hamiltonian stationary Lagrangian translators in the complex plane \cite[Corollary 2,3]{CL10}.

\end{document}